\documentclass[12pt]{article}

\usepackage{latexsym}
\usepackage{graphics}
\usepackage{xspace}
\usepackage{psfrag}
\usepackage{epsfig}
\usepackage{pst-all}
\usepackage{amssymb, amsmath, amsthm, verbatim, mathrsfs}
\usepackage{bbm}

\author{\sf{David Gamarnik} \thanks {Operations Research Center and Sloan School of Management, MIT, Cambridge, MA, 02139, e-mail: \tt{gamarnik@mit.edu} } \and 
\sf{Dmitry Katz} \thanks{T.J. Watson Research Center, IBM, Yorktown Heights, NY, 10598, e-mail:\tt{dimdim@mit.edu}} \and
 \sf{Sidhant Misra} \thanks{Department of Electrical Engineering and Computer Science, MIT, Cambridge, MA, 02139, e-mail: \tt{sidhant@mit.edu}} }

\newtheorem{thm}{Theorem}
\newtheorem{lem}{Lemma}
\newtheorem{cor}{Corollary}
\theoremstyle{definition}
\newtheorem{defin}{Definition}

\newtheorem{assumption}{Assumption}

\begin{document}

\title{Strong spatial mixing for list coloring of graphs}
\maketitle

\abstract
The property of spatial mixing and 
strong spatial mixing in spin systems has been of interest because of its implications on uniqueness of Gibbs measures on infinite graphs and efficient approximation of counting problems that are otherwise known to be
$\#P$ hard. In the context of coloring, strong spatial mixing has been established for regular trees in \cite{GeStefankovic11} when $q \geq \alpha^{*}  \Delta + 1$ where $q$ the number of colors, $\Delta$ is the degree and 
$\alpha^* = 1.763..$ is the unique solution to $xe^{-1/x} = 1$.
It has also been established in
\cite{GoldbergMartinPaterson05} for bounded degree lattice graphs whenever $q \geq \alpha^* \Delta - \beta$ for some constant $\beta$, where $\Delta$ is the maximum vertex degree of the graph. The latter uses a technique based on recursively constructed coupling of Markov chains
 whereas the former is based
on establishing decay of correlations on the tree. We establish strong spatial mixing of list colorings on arbitrary bounded degree 
triangle-free graphs whenever the size of the list of each vertex $v$ is at least $\alpha \Delta(v) + \beta$ where
$\Delta(v)$ is the degree of vertex $v$ and $\alpha > \alpha ^*$ and $\beta$ is a constant that only depends on $\alpha$. We do this by proving the decay of correlations  via recursive contraction of the distance between the marginals measured with respect to a suitably chosen error function.

\begin{section}{Introduction}

	In this paper we study the problem of list colorings of a graph. We explore the strong spatial mixing property of list colorings on triangle-free graphs which pertains to exponential decay of boundary effects when the list 
	coloring is generated uniformly at random conditioned on the coloring of the boundary. This means
	fixing the color of vertices far away from a vertex $v$ has negligible impact (exponentially decaying correlations) on the probability of $v$ being colored with a certain color in its list. 
	A related but weaker notion is weak spatial mixing.
	Strong spatial mixing is stronger than weak spatial mixing because it requires exponential decay of boundary effects even when some of the vertices near $v$ are conditioned to have fixed colors.
	
	Jonasson in \cite{Jonasson02} showed weak spatial mixing on regular trees of any degree $\Delta$ whenever the number of colors $q$ is greater than or equal to $\Delta +1$. However the weakest
	conditions for which strong spatial mixing on trees has been established thus far is by Ge and Stefankovic \cite{GeStefankovic11} for $q \geq \alpha^* \Delta + 1$ where $\alpha^* = 1.763..$ is the unique 
	solution to $xe^{-1/x} = 1$. For lattice graphs (or more generally triangle-free amenable graphs)
	 strong spatial
	mixing was established by Goldberg, Martin and Paterson in \cite{GoldbergMartinPaterson05} for $q \geq \alpha^*  \Delta - \beta$ for a fixed constant $\beta$. In fact their result also holds when the 
	setting is generalized to the list coloring problem. In this paper we generalize these results under a mildly stronger condition by establishing strong spatial mixing of list colorings on arbitrary bounded degree
	 triangle free graphs whenever the size of
	the list of each vertex $v$ is at least $\alpha \Delta(v) + \beta$, where $\Delta(v)$ is the degree of $v$, $\alpha$ satisfies $\alpha > \alpha ^*$ and $\beta$ is a constant that only depends on $\alpha$.
	
	The notion of spatial
	mixing is closely connected to the uniqueness of the infinite volume Gibbs measure on the spin system defined by the list coloring problem. In fact weak spatial mixing is a sufficient condition for there to be  a unique 
	Gibbs measure. Strong spatial mixing is also closely related to the problem of approximately counting the number of valid colorings of a graph which is the partition function of the Gibbs measure. For amenable
	graphs strong spatial mixing also implies rapid mixing
	 of Glauber dynamics which leads to efficient randomized approximation algorithms for computing the total number of valid colorings of a graph, 
	e.g. in \cite{Jerrum95}, \cite{HayesVigoda05}, \cite{Vigoda00}, \cite{Molloy04} etc. 
	The decay of correlations property similar to strong spatial mixing has also been shown to lead to \textit{deterministic} FPTAS (Fully Polynomial Time Approximation Scheme)
	 for computing the partition function of the Gibbs measure. This technique was introduced by Bandyopadhyay and Gamarnik in \cite{BandyopadhyayGamarnik06} (conference version in SODA'06)
	 and Weitz in \cite{Weitz06} and has been subsequently employed by
	 Gamarnik and Katz \cite	{GamarnikKatz08} for the list coloring problem. Since decay of correlations implies the uniqueness of Gibbs measure on regular trees and regular trees represent 
	 maximal growth of the size of the neighborhood for a given degree, it is a general conjecture that efficient 
	 approximability of the counting problem coincides with the uniqueness of Gibbs measure on regular trees. More precisely the conjecture states that 
	 an FPTAS for counting colorings exists for any arbitrary graph whenever $q \geq \Delta + 2$. We are still very far from proving this conjecture or even establishing strong spatial mixing under this condition.
	
	The setup of this paper is similar to \cite{GamarnikKatz08} . In \cite{GamarnikKatz08} it was shown that the logarithm of the ratio of the marginals induced by the two different boundary conditions
	 contract in $\ell_{\infty}$ norm as we move away 
	from the boundary whenever $|L(v)| \geq \alpha \Delta(v) + \beta$ where $\alpha > \alpha^{**} \approx 2.78..$ and $\beta$ is a constant that only depends on $\alpha$.
	In this paper we measure the distance with respect to a conveniently chosen error function which allows us to tighten the contraction argument and relax the 
	required condition to $\alpha > \alpha^* \approx 1.76 ..$. This also means that the Gibbs measure on such graphs is unique. 
	However unlike \cite{GamarnikKatz08} the result of this paper does not directly lead to an FPTAS for counting colorings. 
	We
	give more details about this later in the paper.
	
	The rest of the paper is organized as follows. 
	In Section \ref{sec:def} we introduce the notation, basic definitions and preliminary concepts. Also in this section we provide the statement of our main result and discuss in detail its implications and 
	connections to previous results. In Section \ref{sec:prelim} we establish some preliminary technical results. In Section 
	\ref{sec:general} we prove the main result of this paper. We end the paper with some concluding remarks and discuss directions for future research.

\end{section}


\begin{section}{Definitions and Main Result} \label{sec:def}
	We denote by $\mathcal{G} = (\mathcal{V}, \mathcal{E})$ an infinite graph with the set of vertices and edges given by $\mathcal{V}$ and $\mathcal{E}$. For a fixed vertex $v \in \mathcal{V}$
	we denote by $\Delta(v)$ the degree of $v$ and by $\Delta$ the maximum degree of the graph, i.e. $\Delta = \max_{v \in \mathcal{V}} \Delta(v) < \infty$.
	 The distance between two vertices $v_1$ and $v_2$ in $\mathcal{V}$ is denoted by $d(v_1,v_2)$ which might be infinite if $v_1$ and $v_2$ belong to 
	two different connected components of $\mathcal{G}$. For two finite subsets of vertices $\Psi_1 \subset \mathcal{V}$ and 
	$\Psi_2 \subset \mathcal{V}$, the distance between them is defined as $d(\Psi_1,\Psi_2) = \min \{ d(v_1,v_2) : v_1 \in \Psi_1 , \ v_2 \in \Psi_2\}$.
	 We assume $\{ 1,2, \ldots, q\}$ to be the set of
	 all colors.
	Each vertex $v \in \mathcal{V}$ is associated with a finite list of colors $L(v) \subseteq \{ 1,2, \ldots, q\}$ and $\mathcal{L} = (L(v): v\in \mathcal{V})$ is the sequence of lists. The total variational distance between two discrete
	measures $\mu_1$ and $\mu_2$ on a finite or countable sample space $\Omega$ is given by $||\mu_1 - \mu_2||$ and is defined as $||\mu_1 - \mu_2|| = \sum_{\omega \in \Omega} |\mu_1(\omega) - \mu_2(\omega)|$.

	A valid list coloring $C$ of $\mathcal{G}$ is an assignment to each vertex $v \in \mathcal{V}$, a color $c(v) \in L(v)$ such that no two adjacent vertices have the same color. 
	A measure $\mu$ on the set of all valid colorings of an infinite graph $\mathcal{G}$ is called an infinite volume Gibbs measure with the uniform specification if, for any finite region $\Psi \subseteq \mathcal{G}$,
	the distribution induced on $\Psi$ by $\mu$ conditioned on any coloring $C$ of the vertices $\mathcal{V} \backslash \Psi$ 
	is the uniform conditional distribution on the set of all valid colorings of $\Psi$. We denote this distribution by $\mu_{\Psi}^{C}$. 
	For any finite subset $\Psi \subset\mathcal{G}$, let $\partial \Psi$ denote the boundary of $\Psi$, i.e. the set of vertices which are adjacent to some vertex in $\Psi$ but are not a part of $\Psi$.

	\begin{defin} \label{def:ssm}
		The infinite volume Gibbs measure $\mu$ on $\mathcal{G}$ is said to have strong spatial mixing (with exponentially decaying correlations) if there exists positive constants $A$ and $\theta$ such that 
		for any finite region $\Psi \subset \mathcal{G}$, any two colorings
		$C_1, C_2$ (where 'free' vertices, i.e. vertices to which no color has been assigned, are also allowed) of $\mathcal{V} \backslash \Psi$ which differ only on a subset $W \subset \partial \Psi$,
		and any subset $\Lambda \subseteq \Psi$, 
		\begin{align}
			||\mu_{\Psi}^{C_1} - \mu_{\Psi}^{C_2}||_{\Lambda} \leq A |\Lambda| e^{-\theta d(\Lambda,W)}.
		\end{align}
		Here $||\mu_{\Psi}^{C_1} - \mu_{\Psi}^{C_2}||_{\Lambda}$ denotes the total variational distance between the two distributions $\mu_{\Psi}^{C_1}$ and $\mu_{\Psi}^{C_2}$ restricted to the set $\Lambda$.
	\end{defin}

	We have used the definition of strong spatial mixing from Weitz's PhD thesis \cite{Weitz04}. As mentioned in \cite{Weitz04},
	 this definition of strong spatial mixing is appropriate for general graphs. A similar definition is used in \cite{GoldbergMartinPaterson05}, where the set
	 $\mathcal{W}$ of disagreement was restricted to be a single vertex. This definition is more relevant in the context of lattice graphs (or more generally amenable graphs), where the neighborhood of 
	 a vertex grows slowly with distance from the vertex. In that context, the definition involving one vertex disagreement and the one we have adopted are essentially the same. 
	 	 
	Let $\alpha^{*} = 1.76..$ be the unique root of the equation 
	\begin{align*}
		x e^{-\frac{1}{x}} = 1.
	\end{align*}	
	For our purposes we will assume that the  graph list pair $(\mathcal{G},\mathcal{L})$ satisfies the following.
	\begin{assumption} \label{assume}
		The graph $\mathcal{G}$ is triangle-free.
		The size of the list of each vertex $v$ satisfies
		\begin{align} \label{maincondition}
			& |L(v)| \geq \alpha \Delta(v) + \beta. 
		\end{align}
		for some constant $\alpha > \alpha^*$ and $\beta = \beta(\alpha) \geq \frac{\sqrt{2}}{\sqrt{2} - 1}$ is such that	
		\begin{align*}
			( 1 - 1/\beta ) \alpha e^{-\frac{1}{\alpha} (1+ 1/\beta)} > 1.
		\end{align*}
	\end{assumption}
	
	Using the above assumption we now state our main result.
	\begin{thm} \label{thm:mainresult}
		Suppose Assumption \ref{assume} holds for the graph list pair ($\mathcal{G}, \mathcal{L}$).
		Then the Gibbs measure with the uniform specification on $(\mathcal{G}, \mathcal{L})$ satisfies strong spatial mixing with exponentially decaying correlations.
	\end{thm}
	We establish some useful technical results in the next section before presenting the details of the proof in Section (\ref{sec:general}).
\end{section}


\begin{section}{Preliminary technical results}	\label{sec:prelim}
	The following theorem establishes strong spatial mixing for the special case when $\Lambda$ consists of a single vertex.
	\begin{thm} \label{lem:singlepointssm}
		Suppose Assumption \ref{assume} holds for the graph list pair $(\mathcal{G},\mathcal{L})$. Then 
		there exists positive constants $B$ and $\gamma$ such that given any finite region $\Psi \subset \mathcal{G}$, any two colorings
		$C_1, C_2$ of $\mathcal{V} \backslash \Psi$ which differ only on a subset $W \subseteq \partial \Psi$, 
		and any vertex $v \in \Psi$ and color $j \in L(v)$,
		\begin{align}
			(1 - \epsilon) \leq \frac{ \mathbf{P}(c(v) = j | C_1)} { \mathbf{P}(c(v) = j | C_2)} \leq (1 + \epsilon)
		\end{align} 
		where $\epsilon = B e^{-\gamma d(v,W)}$
	\end{thm}	
	
	We will now show that Theorem \ref{thm:mainresult} follows from Theorem \ref{lem:singlepointssm}. 
	
	\begin{proof}[Proof of Theorem \ref{thm:mainresult}]
	To prove this, we use induction on the size of the subset $\Lambda$. The base case with $|\Lambda| = 1$ is equivalent to the statement of Theorem \ref{lem:singlepointssm}. Assume that the statement of 
	Theorem \ref{thm:mainresult} is true whenever $|\Lambda| \leq t$ for some integer $t \geq1$. We will use this to prove that the statement holds when $|\Lambda| = t+1$. 
	Let the vertices in $\Lambda$ be $v_1,v_2, \ldots, v_{t+1}$. Let $\mathbf{v_k} = (v_1, \ldots, v_k)$ and $\mathbf{J_k} = (j_1, \ldots, j_k)$ where $j_i \in L(v_i), \, 1\leq i \leq k$. Also let $c(\mathbf{v_k}) = \left( c(v_1), c(v_2), \ldots , c(v_k) \right)$ denote the coloring of the vertecies $v_1, v_2, \ldots , v_k$.
	\begin{align*}
		\mathbf{P}(c    (\mathbf{v_{t+1}}) = \mathbf{   J_{t+1}} | C_1) = &\mathbf{P}(c(\mathbf{v_t}) = \mathbf{J_t} | C_1) \mathbf{P}(c(v_{t+1} = j_{t+1} ) |  c(\mathbf{v_t}) = \mathbf{J_t}, C_1 ) \\
		\leq& (1+ \epsilon) \mathbf{P}(c(\mathbf{v_t}) = \mathbf{J_t} | C_1) \mathbf{P}(c(v_{t+1} = j_{t+1} ) |  c(\mathbf{v_t}) = \mathbf{J_t}, C_2 ).	
	\end{align*}
         The inequality in the last statement follows from Theorem \ref{lem:singlepointssm}.
         This gives
         \begin{align*}
         		\mathbf{P}(c(\mathbf{v_{t+1}}) = \mathbf{J_{t+1}} | C_1) &  - \mathbf{P}(c(\mathbf{v_{t+1}}) = \mathbf{J_{t+1}} | C_2) \\
		\leq & (1+ \epsilon) \mathbf{P}(c(\mathbf{v_t}) = \mathbf{J_t} | C_1) \mathbf{P}(c(v_{t+1} = j_{t+1} ) |  c(\mathbf{v_t}) = \mathbf{J_t}, C_2 ) \\
		-& \mathbf{P}(c(\mathbf{v_t}) = \mathbf{J_t} | C_2) \mathbf{P}(c(v_{t+1} = j_{t+1} ) |  c(\mathbf{v_t}) = \mathbf{J_t}, C_2) \\
		=& \epsilon  \mathbf{P}(c(\mathbf{v_t}) = \mathbf{J_t} | C_1) \mathbf{P}(c(v_{t+1} = j_{t+1} ) |  c(\mathbf{v_t}) = \mathbf{J_t}, C_2) \\
		+&  \mathbf{P}(c(v_{t+1} = j_{t+1} ) |  c(\mathbf{v_t}) = \mathbf{J_t}, C_2)  \{  \mathbf{P}(c(\mathbf{v_t}) = \mathbf{J_t} | C_1) - \mathbf{P}(c(\mathbf{v_t}) = \mathbf{J_t} | C_2) \}
         \end{align*}
         Similarly 
         \begin{align*}
         		\mathbf{P}(c(\mathbf{v_{t+1}}) = \mathbf{J_{t+1}} | C_1)  & - \mathbf{P}(c(\mathbf{v_{t+1}}) = \mathbf{J_{t+1}} | C_2) \\
		\geq & - \epsilon  \mathbf{P}(c(\mathbf{v_t}) = \mathbf{J_t} | C_1) \mathbf{P}(c(v_{t+1} = j_{t+1} ) |  c(\mathbf{v_t}) = \mathbf{J_t}, C_2 ) \\
		+ & \mathbf{P}(c(v_{t+1} = j_{t+1} ) |  c(\mathbf{v_t}) = \mathbf{J_t}, C_2)  \{  \mathbf{P}(c(\mathbf{v_t}) = \mathbf{J_t} | C_1) - \mathbf{P}(c(\mathbf{v_t}) = \mathbf{J_t} | C_2) \}
         \end{align*}
         
         Combining the above, we get 
         \begin{align*}
         		  | \mathbf{P}  (c(\mathbf{v_{t+1}}) = \mathbf{J_{t+1}} | &C_1)  -   \mathbf{P}(c(\mathbf{v_{t+1}}) = \mathbf{J_{t+1}} | C_2) | \\
		   \leq& \epsilon  \mathbf{P}(c(\mathbf{v_t}) = \mathbf{J_t} | C_1)
		 \mathbf{P}(c(v_{t+1} = j_{t+1} ) |  c(\mathbf{v_t}) = \mathbf{J_t}, C_2)  \\
		+ &\mathbf{P}(c(v_{t+1} = j_{t+1} ) |  c(\mathbf{v_t}) = \mathbf{J_t}, C_2)  \  | \{ \mathbf{P}(c(\mathbf{v_t}) = \mathbf{J_t} | C_1) - \mathbf{P}(c(\mathbf{v_t}) = \mathbf{J_t} | C_2) \} |
         \end{align*}
         We can now bound the total variational distance $||\mu_{\Psi}^{C_1} - \mu_{\Psi}^{C_2}||_{\Lambda}$ as follows.
   	\begin{align*}
		\sum_{j_i \in L(v_i), \  1 \leq i \leq t+1} & | \mathbf{P}(c(\mathbf{v_{t+1}}) = \mathbf{J_{t+1}} | C_1)  - \mathbf{P}(c(\mathbf{v_{t+1}}) = \mathbf{J_{t+1}} | C_2)  | \\
		\leq &\sum_{j_i \in L(v_i), \  1 \leq i \leq t+1} \epsilon  \mathbf{P}(c(\mathbf{v_t}) = \mathbf{J_t} | C_1)
		 \mathbf{P}(c(v_{t+1} = j_{t+1} ) |  c(\mathbf{v_t}) = \mathbf{J_t}, C_2 ) \\
		  +&  \sum_{j_i \in L(v_i), \  1 \leq i \leq t+1} \mathbf{P}(c(v_{t+1} = j_{t+1} ) |  c(\mathbf{v_t}) = \mathbf{J_t}, C_2)  \  | \{ \mathbf{P}(c(\mathbf{v_t}) = \mathbf{J_t} | C_1) - \mathbf{P}(c(\mathbf{v_t}) = \mathbf{J_t} | C_2) \} | \\
		 \leq & \epsilon + ||\mu_{\Psi}^{C_1} - \mu_{\Psi}^{C_2}||_{\Lambda \backslash v_{t+1}}		\\
		 \leq & (t+1) \epsilon.
	\end{align*}  
	where the last statement follows from the induction hypothesis. This completes the induction argument.
	\end{proof}

	So, in order to establish Theorem \ref{thm:mainresult}, it is enough to show that Theorem \ref{lem:singlepointssm} is true. We claim that Theorem \ref{lem:singlepointssm} follows from the theorem below which
	establishes weak spatial mixing whenever Assumption \ref{assume} holds. In other words under Assumption \ref{assume}, strong spatial mixing of list colorings for marginals of a single vertex holds whenever
	weak spatial mixing holds. In fact $\mathcal{G}$ need not be triangle-free for this implication to be true as will be clear from the proof below. 
	\begin{thm} \label{lem:wsm}
		Suppose Assumption \ref{assume} holds for the graph list pair $(\mathcal{G},\mathcal{L})$.Then there exist positive constants $B$ and $\gamma$ such that 
		given any finite region $\Psi \subset \mathcal{G}$, any two colorings
		$C_1, C_2$ of $\mathcal{V} \backslash \Psi$,
		\begin{align}
			(1 - \epsilon) \leq \frac{ \mathbf{P}(c(v) = j | C_1)} { \mathbf{P}(c(v) = j | C_2)} \leq (1 + \epsilon)
		\end{align} 
		where $\epsilon = B e^{-\gamma d(v,\partial \Psi)}$.
	\end{thm}
	We first show how Theorem \ref{lem:singlepointssm} follows from Theorem \ref{lem:wsm}.
	\begin{proof}[Proof of Theorem \ref{lem:singlepointssm}]
		Consider two colorings $C_1, C_2$ of the boundary $\partial \Psi$ of $\Psi$ which differ only on a subset $W \subseteq \partial \Psi$ as in the statement of
		 Theorem \ref{lem:singlepointssm}. Let $d = d(v,W)$.
		 We first construct a new graph list pair $(\mathcal{G}', \mathcal{L}')$ from $(\mathcal{G},\mathcal{L})$. Here
		$\mathcal{G}'$ is obtained from $\mathcal{G}$ by deleting all vertices in $\partial \Psi$ which are at a distance less than $d$ from $v$. Notice that for all such vertices
		$C_1$ and $C_2$ agree. Whenever a vertex $u$ is deleted from $\mathcal{G}$, remove from the lists
		of the neighbors of $u$ the color $c(u)$ which is the color of $u$ under both $C_1$ and $C_2$.
		 This defines the new list $\mathcal{L}'$. In this process, whenever a vertex $u$ loses a color in its 
		list it also loses one of its edges. Also for $\alpha > \alpha^* > 1$, we have $|L(v)| -1 \geq \alpha (\Delta(v)-1)  + \beta$ whenever $|L(v)| \geq \alpha \Delta(v) +\beta$. 
		Therefore, the new graph list pair $(\mathcal{G}',\mathcal{L}')$
		also satisfies Assumption \ref{assume}.
		Define the region $\Psi' \subset \mathcal{G}'$ as the ball of radius $(d-1)$ centered at $v$. Let $D_1$ and $D_2$ be two colorings of $(\Psi')^c$ which agree with $C_1$ and $C_2$ respectively.
		From the way in which $\mathcal{G}',\mathcal{L}'$ is constructed we have
		\begin{align} \label{probs}
			\mathbf{P}_{\mathcal{G},\mathcal{L}}(c(v) = j | C_i) = \mathbf{P}_{\mathcal{G}',\mathcal{L}'}(c(v) = j | D_i) \quad \mbox{for } i=1,2 .  
		\end{align}
		where $\mathbf{P}_{\mathcal{G},\mathcal{L}}(\mathbf{E})$ denotes the probability of the event $\mathbf{E}$ in the graph list pair ($\mathcal{G},\mathcal{L}$).
		If $\mathcal{V}'$ is the set of all vertices of $\mathcal{G}'$, then $D_1$ and $D_2$ assign colors only to vertices in $\mathcal{V'} \backslash \Psi'$.
		So we can apply Theorem \ref{lem:wsm} for the region $\Psi'$ and the proof is complete.
	\end{proof}
	
	So it is sufficient to prove Theorem \ref{lem:wsm} which we defer till section (\ref{sec:general}). We use the rest of this section to discuss some implications 
	of our result and connections between our result and previous established results for strong spatial mixing for coloring of graphs.
	
	The statement in Lemma {\ref{lem:wsm}} is what is referred to as weak spatial mixing \cite{Weitz04}. In general weak spatial mixing is a weaker condition and does not imply strong spatial mixing. This is indeed the case
	when we consider the \textit{coloring} problem of a graph $\mathcal{G}$ by $q$ colors, i.e. the case when the lists $L(v)$ are the same for all $v \in \mathcal{V}$. However, interestingly, 
	as the above argument shows, for the case of 
	\textit{list coloring} strong spatial mixing follows from weak spatial mixing when the graph list pair satisfies Assumption \ref{assume}.	
		
	We observed that the strong spatial mixing result for amenable graphs in \cite{GoldbergMartinPaterson05} also extends to the case of list colorings. Indeed the proof technique only requires a
	local condition similar to that in Assumption \ref{assume} that we have adopted as opposed to a global condition like $q \geq \alpha \Delta + \beta$.
	 Also in \cite{GoldbergMartinPaterson05}, the factor $|\Lambda|$ in the definition (\ref{def:ssm}) was shown to be not necessary which makes their
	statement stronger. We show that this stronger statement is also implied by our result. In particular, assuming Theorem \ref{lem:singlepointssm} is true, we prove the following corollary.
	\begin{cor} \label{lem:gmpdiscussion}
		Suppose the graph list pair $(\mathcal{G},\mathcal{L})$ satisfies Assumption \ref{assume}. Then there exists positive constants $A$ and $\theta$ such that 
		given any finite region $\Psi \subset \mathcal{G}$, any two colorings
		$C_1, C_2$ of the boundary $\partial \Psi$ of $\Psi$ which differ at only one point $f \in \partial \Psi$,
		and any subset $\Lambda \subseteq \Psi$, 
		\begin{align}
			||\mu_{\Psi}^{C_1} - \mu_{\Psi}^{C_2}||_{\Lambda} \leq A e^{-\theta d(\Lambda,f)}.
		\end{align}
	\end{cor}
	
	\begin{proof}
	Let the color of $f$ be $j_1$ in $C_1$ and $j_2$ in $C_2$.
	Let $\mathcal{C}(\Lambda)$ be the set of all possible colorings of the set $\Lambda$.
		\begin{align*}
			||\mu_{\Psi}^{C_1} - \mu_{\Psi}^{C_2}||_{\Lambda}  &= \sum_{\sigma \in \mathcal{C}(\Lambda)} |\mathbf{P}(\sigma | c(f) = j_1) - \mathbf{P}(\sigma | c(f) = j_2) | \\
			&=  \sum_{\sigma \in \mathcal{C}(\Lambda)} \left|    \frac{ \mathbf{P}(c(f) = j_1 | \sigma)}{\mathbf{P}(c(f) = j_1)} \mathbf{P}(\sigma)  - \frac{ \mathbf{P}(c(f) = j_2 | \sigma)}{\mathbf{P}(c(f) = j_2)} \mathbf{P}(\sigma) \right|
		\end{align*}	
		For any $j \in L(f)$, using Theorem \ref{lem:singlepointssm} we have for $\epsilon = A e^{-\theta d(\Lambda,f)} $,
		\begin{align*}
			 \frac{ \mathbf{P}(c(f) = j | \sigma)}{\mathbf{P}(c(f) = j)} &=  \frac{ \mathbf{P}(c(f) = j | \sigma)}{   \sum_{\sigma' \in \mathcal{C}(\Lambda)} \mathbf{P}(c(f) = j| \sigma') \mathbf{P}(\sigma')} \\
			   &=  \frac{  \sum_{\sigma' \in \mathcal{C}(\Lambda)} \mathbf{P}(c(f) = j| \sigma) \mathbf{P}(\sigma')}{   \sum_{\sigma' \in \mathcal{C}(\Lambda)} \mathbf{P}(c(f) = j| \sigma') \mathbf{P}(\sigma')} \\
			   & \leq \frac{  \sum_{\sigma' \in \mathcal{C}(\Lambda)} \mathbf{P}(c(f) = j| \sigma') (1 + \epsilon) \mathbf{P}(\sigma')}{   \sum_{\sigma' \in \mathcal{C}(\Lambda)} \mathbf{P}(c(f) = j| \sigma') \mathbf{P}(\sigma')} \\
			   &= 1 + \epsilon.
		\end{align*}
		Similarly we can also prove for any $j \in L(f)$ 
		\begin{align*}
			\frac{ \mathbf{P}(c(f) = j | \sigma)}{\mathbf{P}(c(f) = j)} \geq 1 - \epsilon.
		\end{align*}
		Therefore,
		\begin{align*}
			||\mu_{\Psi}^{C_1} - \mu_{\Psi}^{C_2}||_{\Lambda} \leq  \sum_{\sigma \in \mathcal{C}(\Lambda)} |(1 + \epsilon) - (1 - \epsilon)  | \mathbf{P}(\sigma) = 2 \epsilon.
		\end{align*}
		
	\end{proof}

	The notion of strong spatial mixing we have adopted also implies the uniqueness of Gibbs measure on the spin system described by the list coloring problem. In fact \textit{Weak Spatial Mixing} described in Theorem
	\ref{lem:wsm} is sufficient for the uniqueness of Gibbs measure (see Theorem 2.2 and the discussion following Definition 2.3 in \cite{Weitz04}). We summarize this in the corollary that follows.
	
	 \begin{cor}
	 	Suppose the graph list pair $\mathcal{G},\mathcal{L}$ satisfy Assumption \ref{assume}. Then the infinite volume Gibbs measure on the list colorings of $\mathcal{G}$ is unique. 
	 \end{cor}

\end{section}
 	

 \begin{section}{Proof of Theorem \ref{lem:wsm}} \label{sec:general}	
 		Let $v \in \mathcal{V}$ be a fixed vertex of $\mathcal{G}$. Let $m = \Delta(v)$ denote the degree of $v$ and let $v_1, v_2, \ldots, v_m$ be the neighbors of $v$. The statement of the theorem is trivial if $m=0$
		 ($v$ is an isolated vertex).
		 Let $q_v = |L(v)|$ and $q_{v_i} = |L(v_i)|$. Also let $\mathcal{G}_v$ be the graph 	obtained from $\mathcal{G}$ by deleting the vertex $v$. We begin by proving two useful recursions on the 	
		 marginal probabilities in the following lemmas.
	\begin{lem} \label{lem:normalizedRecursion}
		Let $j_1,j_2 \in L(v)$. Let $\mathcal{L}_{i,j_1,j_2}$ denote the list associated with graph $\mathcal{G}_v$ which is
		obtained from $\mathcal{L}$ by removing the color $j_1$ from 
		the lists $L(v_k)$ for $k < i$ and removing the color $j_2$ from the lists $L(v_k)$ for $k>i$ (if any of these lists do not contain the respective color then no change is made to them). Then we have
		\begin{align*}
			\frac{\mathbf{P}_{\mathcal{G},\mathcal{L}}(c(v) = j_1)}{\mathbf{P}_{\mathcal{G},\mathcal{L}}(c(v) = j_2)} = \prod_{i=1}^{m} \frac{1 - \mathbf{P}_{\mathcal{G}_v,L_{i,j_1,j_2}} (c(v_i) = j_1)}{1 - \mathbf{P}_{\mathcal{G}_v,L_{i,j_1,j_2}} (c(v_i) = j_2)}
		\end{align*}
	\end{lem}
		
	\begin{proof}
		Let $Z_{\mathcal{G},\mathcal{L}}(\mathcal{M})$ denote the number of colorings of a finite graph $\mathcal{G}$ with the condition $\mathcal{M}$ satisfied. For example, $Z_{\mathcal{G},\mathcal{L}}(c(v) = j)$
		denotes the number of valid colorings of $\mathcal{G}$ when the color of $v$ is fixed to be $j \in L(v)$.
		We use a telescoping product argument to prove the lemma:
		\begin{align*}
			\frac{\mathbf{P}_{\mathcal{G},\mathcal{L}}(c(v) = j_1)}{\mathbf{P}_{\mathcal{G},\mathcal{L}}(c(v) = j_2)} &= \frac{Z_{\mathcal{G},\mathcal{L}}(c(v) = j_1)}{Z_{\mathcal{G},\mathcal{L}}(c(v) = j_2)} \\
			&= \frac{Z_{\mathcal{G}_v,\mathcal{L}}(c(v_i) \neq j_1, \ 1 \leq i \leq m)}   {  Z_{\mathcal{G}_v,\mathcal{L}}(c(v_i) \neq j_2, \ 1 \leq i \leq m)} \\
			&=  \frac{\mathbf{P}_{\mathcal{G}_v,\mathcal{L}}(c(v_i) \neq j_1, \ 1 \leq i \leq m)}   {  \mathbf{P}_{\mathcal{G}_v,\mathcal{L}}(c(v_i) \neq j_2, \ 1 \leq i \leq m)} \\
			&= \prod_{i=1}^{m} \frac{\mathbf{P}_{\mathcal{G}_v,\mathcal{L}}(c(v_k) \neq j_1, \ 1 \leq k \leq i, \ c(v_k) \neq j_2, \ i+1 \leq k \leq m)}   {  \mathbf{P}_{\mathcal{G}_v,\mathcal{L}}(c(v_k) \neq j_1, \ 1 \leq k \leq i-1, \ c(v_k) \neq j_2, \ i \leq k \leq m)} \\
			&= \prod_{i=1}^{m} \frac{\mathbf{P}_{\mathcal{G}_v,\mathcal{L}}(c(v_i) \neq j_1 \ | c(v_k) \neq j_1, \ 1 \leq k \leq i-1, \ c(v_k) \neq j_2, \ i+1 \leq k \leq m)}   {  \mathbf{P}_{\mathcal{G}_v,\mathcal{L}}(c(v_i) \neq j_2 \ | c(v_k) \neq j_1, \ 1 \leq k \leq i-1, \ c(v_k) \neq j_2, \ i+1 \leq k \leq m)} \\
			&= \prod_{i=1}^{m} \frac{1 - \mathbf{P}_{\mathcal{G}_v,L_{i,j_1,j_2}} (c(v_i) = j_1)}{1 - \mathbf{P}_{\mathcal{G}_v,L_{i,j_1,j_2}} (c(v_i) = j_2)}.
		\end{align*}
	\end{proof}	
	
	The following lemma was proved in \cite{GamarnikKatz08}. We provide the proof here for completeness.
	\begin{lem} \label{lem:recursion}
		Let $j \in L(v)$. Let $\mathcal{L}_{i,j}$ denote the list associated with the graph $\mathcal{G}_v$ which is 
		obtained from $\mathcal{L}$ by 
		removing the color $j$ (if it exists) from the list $L(v_k)$ for $k<i$. Then we have
		\begin{align*}
			\mathbf{P}_{\mathcal{G},\mathcal{L}}(c(v) = j) = \frac{\prod_{i=1}^{m} \left(  1 - \mathbf{P}_{\mathcal{G}_v,\mathcal{L}_{i,j}}(c(v_i) = j)  \right) }{ \sum_{k \in L(v)} \prod_{i=1}^{m} \left(  1 - \mathbf{P}_{\mathcal{G}_v,\mathcal{L}_{i,k}}(c(v_i) = k)  \right) }.
		\end{align*}
	\end{lem}
	\begin{proof}
	\begin{align}
		\mathbf{P}_{\mathcal{G},\mathcal{L}}(c(v) = j) &= \frac{Z_{\mathcal{G},\mathcal{L}}  (c(v) = j)}{ \sum_{k \in  L(v)} Z_{\mathcal{G},\mathcal{L}}(c(v) = k)}    \notag \\
		&= \frac{  Z_{\mathcal{G}_v, \mathcal{L}}  (c(v_i)  \neq j, \ 1 \leq i \leq m)  }{  \sum_{k \in L(v)} Z_{\mathcal{G}_v, \mathcal{L}}  (c(v_i)  \neq k, \ 1 \leq i \leq m) }   \notag \\
		& = \frac{  \mathbf{P}_{\mathcal{G}_v, \mathcal{L}}  (c(v_i)  \neq j, \ 1 \leq i \leq m)  }{  \sum_{k \in L(v)} \mathbf{P}_{\mathcal{G}_v, \mathcal{L}}  (c(v_i)  \neq k, \ 1 \leq i \leq m) }    \label{interprod}
	\end{align}	
	Now for any $k \in L(v)$, 
	\begin{align*}
		\mathbf{P}_{\mathcal{G}_v, \mathcal{L}}  (c(v_i)  \neq k, \ 1 \leq i \leq m) &= \mathbf{P}_{\mathcal{G}_v,\mathcal{L}}  (c(v_1) \neq k)  \prod_{i=2}^{m} \mathbf{P}_{\mathcal{G}_v, \mathcal{L}}
		  ( c(v_i) \neq k  |  c(v_l) \neq k, \ 1 \leq l \leq k-1  ). \\
		&= \prod_{i=1}^{m} \mathbf{P}_{\mathcal{G}_v, \mathcal{L}_{i,k}} (c(v_i) \neq k).
	\end{align*}
	Substituting this into (\ref{interprod}) completes the proof of the lemma.
	\end{proof}
	
	Before proceeding to the proof of Theorem \ref{lem:wsm}, we first establish upper and lower bounds on the marginal probabilities associated with vertex $v$.
	
	\begin{lem} \label{lem:probub}
		For every $j \in L(v)$ and for $l = 1,2$ the following bounds hold.
		\begin{align}
			&\mathbf{P}(c(v) = j | C_l) \leq 1/\beta.  \label{easyub} \\
			&\mathbf{P}(c(v) = j | C_l) \leq { \left( m \alpha e^{-\frac{1}{\alpha} (1+ 1/\beta)} \right) }^{-1}     \label{ub} \\
			&\mathbf{P}(c(v) = j | C_l) \geq q^{-1} (1 - 1/\beta)^{\Delta}. \label{lb}
		\end{align}
		\end{lem}
	\begin{proof}	These bounds were proved in \cite{GamarnikKatz08} with a different constant, i.e. $\alpha > \alpha^{**} \approx 2.84$. Here we prove the bound when Assumption \ref{assume} holds.
				In this proof we assume $l=1$. The case $l=2$ follows by an identical argument.
				Let $\mathcal{C}$ denote the set of all possible colorings of the 
				children $v_1, \ldots, v_m$ of $v$. Note that for any $\mathbf{c} \in \mathcal{C}$, $\mathbf{P}(c(v) = j | \mathbf{c}) \leq \frac{1}{|L(v)| - \Delta(v)} \leq 1/\beta$.
				and (\ref{easyub}) follows.
							
				To prove (\ref{ub}) we will show that for every coloring of the neighbors of the neighbors of $v$, the bound is satisfied.
				So, first fix a coloring $\mathbf{c}$ of the vertices at distance two from $v$. Conditioned on this coloring, define for $j \in L(v)$ the marginal
				\begin{align*}
					t_{ij} = \mathbf{P}_{\mathcal{G}_v,\mathcal{L}_{i,j}}(c(v_i) = j | \mathbf{c}). 
				\end{align*}
				Note that by (\ref{easyub}) we have $t_{ij} \leq 1/\beta$.
				Because $\mathcal{G}$ is triangle-free, there are no edges between the neighbors of $v$ and once we condition on $\mathbf{c}$, we have
				\begin{align*}
					 \mathbf{P}_{\mathcal{G}_v,\mathcal{L}_{i,j}}(c(v_i) = j | \mathbf{c}) =  \mathbf{P}_{\mathcal{G}_v,\mathcal{L}}(c(v_i) = j | \mathbf{c}).
				\end{align*}
				So we obtain
				\begin{align}
				 	\sum_{j \in L(v)} t_{ij} = \sum_{j \in L(v) \bigcap L(v_i)}    \mathbf{P}_{\mathcal{G}_v,\mathcal{L}}(c(v_i) = j | \mathbf{c})   \leq 1. \label{lessthanone}
				 \end{align}
				From Lemma \ref{lem:recursion} we have 
				\begin{align}
					\mathbf{P}(c(v) = j | C_1) = \frac{\prod_{i=1}^{m} \left(  1 -  t_{ij}  \right) }{ \sum_{k \in L(v)} \prod_{i=1}^{m} \left(  1 - t_{ik} \right) }   \label{ubinter}
					 \leq \frac{1}{    \sum_{k \in L(v)} \prod_{i=1}^{m} \left(  1 - t_{ik} \right) }. 
 				\end{align}
				Using Taylor expansion for $\log(1-x)$, we obtain
				\begin{align*}
					\prod_{i=1}^{m} (1-t_{ik}) = \prod_{i=1}^{m} e^{\log(1 - t_{ik})} = \prod_{i=1}^{m} e^{-t_{ik} - \frac{1}{2(1 - \theta_{ik})^2}t_{ik}^2}.
				\end{align*}
				where $0 \leq \theta_{ik} \leq t_{ik}$. So $\theta_{ik}$ satisfies $(1 - \theta_{ik})^2 \geq (1 - 1/\beta)^2 \geq 1/2$ by Assumption \ref{assume}. Thus, we obtain
				\begin{align*}
					\prod_{i=1}^{m} (1-t_{ik}) \geq \prod_{i=1}^{m} e^{- (1 + 1/\beta) t_{ik}} = e^{-(1 + 1/\beta) \sum_{i=1}^{m} t_{ik}}
				\end{align*}
				Using the fact that arithmetic mean is greater than geometric mean and using (\ref{lessthanone}), we get
				\begin{align*}
					 \sum_{k \in L(v) } \prod_{i=1}^{m} (1-t_{ik}) &\geq q_v \left( \prod_{k \in L(v)} \prod_{i=1}^{m} (1-t_{ik}) \right)^{1/q_v} \\
					 &\geq q_v exp \left( {- q_v^{-1} (1+1/\beta) \sum_{i=1}^{m} \sum_{k \in L(v)} t_{i,k}  } \right)\\
					 &\geq (\alpha m + \beta) e^{-(1+1/\beta)\frac{m}{\alpha m + \beta}} \\
					 & \geq \alpha m e^{-(1+1/\beta) \frac{1}{\alpha}}.
				\end{align*}
				Combining with (\ref{ubinter}) the proof of (\ref{ub}) is now complete.
				
				From (\ref{easyub}) we have $\prod_{i=1}^{m} (1-t_{ij}) \geq (1 - 1/\beta)^{m} \geq (1 - 1/\beta)^{\Delta}$. Also
				$ \sum_{k \in L(v)} \prod_{i=1}^{m} \left(  1 - t_{ik} \right) \leq q_v \leq q$.
				So we have
				\begin{align*}
					\mathbf{P}(c(v) = j | C_1) =   \frac{\prod_{i=1}^{m} \left(  1 -  t_{ij}  \right) }{ \sum_{k \in L(v)} \prod_{i=1}^{m} \left(  1 - t_{ik} \right) } \geq q^{-1} (1 - 1/\beta)^{\Delta}.
				\end{align*}
	\end{proof}

	 For $j \in L(v)$, define
	\begin{align*}
		x_j &= \mathbf{P}_{\mathcal{G},\mathcal{L}}(c(v) = j | C_1), \\
		y_j &= \mathbf{P}_{\mathcal{G},\mathcal{L}}(c(v) = j | C_2),
	\end{align*}
	and the vector of marginals
	\begin{align*}
		\mathbf{x} &= \left( x_j : j \in L(v) \right), \\
		\mathbf{y} &= \left( y_j : j \in L(v) \right).
	\end{align*}

	We define a suitably chosen error function which we will use to establish decay of correlations and prove Theorem \ref{lem:wsm}. This error function $\mathbf{E}(\mathbf{x},\mathbf{y})$ is defined as
	\begin{align*}
		\mathbf{E}(\mathbf{x},\mathbf{y}) = \max_{j \in L(v)}  \log\left(  \frac{x_j}{y_j} \right)  - \min_{j \in L(v)}  \log\left(  \frac{x_j}{y_j} \right).
	\end{align*}
	By (\ref{lb}) of Lemma \ref{lem:probub} we have $x_j, y_j >0$ for $j \in L(v)$. So the above expression is well-defined.
	Let $j_1 \in L(v)$ ($j_2 \in L(v)$) achieve the maximum (minimum) in the above expression.
		Recall that for given $j_1,j_2$, we denote by $\mathcal{L}_{i,j_1,j_2}$ the list associated with graph $\mathcal{G}_v$ which is
		obtained from $\mathcal{L}$ by removing the color $j_1$ from 
		the lists $L(v_k)$ for $k < i$ and removing the color $j_2$ from the lists $L(v_k)$ for $k>i$.
		Define for each $1 \leq i \leq m$ and $j \in L_{i,j_1,j_2}(v_i)$ the marginals
		\begin{align*}
			x_{ij} &= \mathbf{P}_{\mathcal{G}_v,\mathcal{L}_{i,j_1,j_2}} (c(v_i) = j | C_1) , \\
			y_{ij} &= \mathbf{P}_{\mathcal{G}_v,\mathcal{L}_{i,j_1,j_2}} (c(v_i) = j | C_2).
		\end{align*}
		and the corresponding vector of marginals
		\begin{align*}
			\mathbf{x_i} &= \left( x_{ij} : j \in L_{i,j_1,j_2}(v_i) \right), \\
			\mathbf{y_i} &= \left( y_{ij} : j \in L_{i,j_1,j_2}(v_i) \right).	
		\end{align*}

	First we prove the following useful fact regarding the terms appearing in the definition of the error function.
	\begin{lem} \label{lem:signs}
		With $x_j$ and $y_j$ defined as before, we have
		\begin{align*}
			&\max_{j \in L(v)}  \log\left(  \frac{x_j}{y_j} \right) \geq 0, \\
			&\min_{j \in L(v)}  \log\left(  \frac{x_j}{y_j} \right) \leq 0
		\end{align*}	
	\end{lem}
	
	\begin{proof}
		We have $\sum_{j \in L(v)} (x_j - y_j) = \sum_{j \in L(v)} x_j - \sum_{j\in L(v)} y_j = 0.$
		This gives
		\begin{align*}
			& q_v \max_{j \in L(v)} (x_j - y_j) \geq \sum_{j \in L(v)} {x_j - y_j} = 0 \\
			\mbox{which implies} \quad & \max_{j \in L(v)}  \log\left(  \frac{x_j}{y_j} \right) \geq 0.
		\end{align*}
		Similarly, 
		\begin{align*}
			&q_v \min_{j \in L(v)} (x_j - y_j) \leq \sum_{j \in L(v)} x_j - y_j= 0 \\
			\mbox{which implies} \quad  &\min_{j \in L(v)}  \log\left(  \frac{x_j}{y_j} \right) \leq 0.
		\end{align*}
	\end{proof}

	We are now ready to prove the following key result which shows that the distance between the marginals induced by the two different boundary conditions measured with respect to the metric defined by the error function
	 $\mathbf{E}(\mathbf{x},\mathbf{y})$ contracts.
	 Let $\epsilon \in (0,1)$ be such that
	 \begin{align*}
		(1-\epsilon) = \frac{1}{(1 - 1/\beta)\alpha e^{-\frac{1}{\alpha} (1+ 1/\beta)} }.
	\end{align*}	 
	Assumption \ref{assume} guarantees that such an $\epsilon$ exists.
 	\begin{lem} \label{lem:generalcontraction}
	Let $m_i = \Delta_{\mathcal{G}_v}(v_i)$. Then
	\begin{align} \label{eq:generalcontraction}
		\frac{1}{m} \mathbf{E}(\mathbf{x},\mathbf{y}) \leq (1-\epsilon) \max_{i:m_i>0} \frac{1}{m_i} \mathbf{E}(\mathbf{x_i},\mathbf{y_i}).
	\end{align}
	The expression on the right hand side of (\ref{eq:generalcontraction}) is interpreted to be $0$ if $m_i=0$ for all $i$.
	\end{lem}
	
	\begin{proof}
		If $j_1 = j_2$, then $\mathbf{E}(\mathbf{x},\mathbf{y})  = 0$.	
		Otherwise,
		\begin{align*}
			\mathbf{E}(\mathbf{x},\mathbf{y}) &= \log\left(  \frac{x_{j_1}}{y_{j_1}} \right)  - \log\left(  \frac{x_{j_2}}{y_{j_2}} \right) \\
			&= \log\left(  \frac{x_{j_1}}{x_{j_2}} \right)  - \log\left(  \frac{y_{j_1}}{y_{j_2}} \right). \\
		\end{align*}
		Introduce the following variables:
		\begin{align*}
			z_{ij} &= \log \left( x_{ij} \right) \\
			w_{ij} &= \log \left(  y_{ij} \right) 
		\end{align*}
		Using the recursion in Lemma \ref{lem:normalizedRecursion} we have
		\begin{align*}
			\mathbf{E}(\mathbf{x},\mathbf{y}) &= \log \left( \prod_{i=1}^{m}  \frac{1 - e^{z_{i j_1}}}{1 - e^{z_{i j_2}}} \right) - \log \left( \prod_{i=1}^{m}  \frac{1 - e^{w_{i j_1}}}{1 - e^{w_{i j_2}}} \right) \\
			&= \left[ \sum_{i=1}^{m} \log \left( 1 - e^{z_{i j_1}} \right) - \log \left( 1 - e^{w_{i j_1}} \right) \right] - \left[ \sum_{i=1}^{m} \log \left( 1 - e^{z_{i j_2}} \right) - \log \left( 1 - e^{w_{i j_2}} \right) \right].
		\end{align*}
		For $j = j_1,j_2$ let
		\begin{align*}
			z_j &=  \sum_{i=1}^{m} \log \left( 1 - e^{z_{i j}} \right) \\
			w_j &=  \sum_{i=1}^{m} \log \left( 1 - e^{w_{i j}} \right).
		\end{align*}
		Then we can rewrite $\mathbf{E}(\mathbf{x},\mathbf{y})$ as 
		\begin{align}
			\mathbf{E}(\mathbf{x},\mathbf{y}) = (z_{j_1} - w_{j_1}) - (z_{j_2} - w_{j_2}).    \label{rewriteerror}
		\end{align}	
			Define the continuous function $f : [0,1] \rightarrow \mathbbm{R}$ as
		\begin{align*}
			f(t) = \sum_{i=1}^{m} \log( 1 - e^{z_{ij}}) - \sum_{i=1}^{m} \log( 1 - e^{z_{ij} + t(w_{ij}- z_{ij})} ).
		\end{align*}
		Then we have $f(0) = 0$ and $f(1) = z_j - w_j$. Applying the mean value theorem, there exists $t \in (0,1)$ such that
		\begin{align*}
			z_j - w_j = f(1) - f(0) = f'(t).
		\end{align*}
		Computing the expression for $f(t)$, we get
		\begin{align*}
			z_j - w_j = \sum_{i=1}^{m} \frac{e^{z_{ij} + t(w_{ij}- z_{ij})}}{1 - e^{z_{ij} + t(w_{ij}- z_{ij})}} (w_{ij} - z_{ij}).
		\end{align*}
		Observe that if  $j \notin L_{i,j_1,j_2}(v_i)$, then $\mathbf{P}_{\mathcal{G}_v,L_{i,j_1,j_2}} (c(v_i) = j | C_1) = \mathbf{P}_{\mathcal{G}_v,L_{i,j_1,j_2}} (c(v_i) = j | C_2) = 0$.
		Hence for $j \notin L_{i,j_1,j_2}(v_i)$, we have $w_{ij}=z_{ij}$. Also if $m_i=0$ then $v$ is an isolated vertex in $\mathcal{G}_v$ and in this case we also have $w_{ij}=z_{ij}$. Using this fact, we have
		\begin{align*}
			z_j - w_j = \sum_{   i:m_i >0 }   \mathbf{1}_{ \{  j \in L_{i,j_1,j_2}(v_i)    \}}   \left( \frac{e^{z_{ij} + t(w_{ij}- z_{ij})}}{1 - e^{z_{ij} + t(w_{ij}- z_{ij})}}    \right) (w_{ij} - z_{ij})   .
		\end{align*}
			From convexity of $e^x$ and Lemma \ref{lem:probub} we have
	\begin{align*}
		0 < e^{z_{ij} + t(w_{ij}- z_{ij})} \leq  t e^{w_{ij}} + (1-t) e^{z_{ij}} \leq  { \left( m_i \alpha e^{-\frac{1}{\alpha} (1+ 1/\beta)} \right) }^{-1}. 
	\end{align*}
	Similarly, again using Lemma \ref{lem:probub} we have
	\begin{align*}
		0 < \frac{1}{1 - e^{z_{ij} + t(w_{ij}- z_{ij})} } \leq \frac{1}{1 - 1/\beta}.
	\end{align*}
	Combining we have for $j = j_1, j_2$,
	\begin{align*}
		0 < \frac{e^{z_{ij} + t(w_{ij}- z_{ij})}}{1 - e^{z_{ij} + t(w_{ij}- z_{ij})}} \leq \frac{1}{  (1 - 1/\beta) \left(   m_i \alpha e^{-\frac{1}{\alpha} (1+ 1/\beta)} \right)  } = \frac{1 - \epsilon}{m_i}.
	\end{align*}
	From Lemma \ref{lem:signs} we have $\max_{k \in L_{i,j_1,j_2}(v_i)} \{ w_{ik} - z_{ik} \} \geq 0 $ and  $\min_{k \in L_{i,j_1,j_2}(v_i)} \{w_{ik} - z_{ik} \} \leq 0 $.
		Using this
		\begin{align*}
			z_{j_1} - w_{j_1} \leq \sum_{i:m_i>0} \frac{1 - \epsilon}{m_i} \max_{k \in  L_{i,j_1,j_2}(v_i)} \{w_{ik} - z_{ik}\},
		\end{align*}
		and
		 \begin{align*}
			z_{j_2} - w_{j_2} \geq \sum_{i:m_i>0} \frac{1 - \epsilon}{m_i} \min_{k \in  L_{i,j_1,j_2}(v_i)} \{w_{ik} - z_{ik}\}.
		\end{align*}
		By using the above bounds in (\ref{rewriteerror}) we get 
		\begin{align}
			\mathbf{E}(\mathbf{x},\mathbf{y}) &\leq \sum_{i:m_i>0} \left( \frac{1- \epsilon}{m_i} \right) \left[ \max_{k \in  L_{i,j_1,j_2}(v_i)} \{w_{ik} - z_{ik}\} - \min_{k \in  L_{i,j_1,j_2}(v_i)} \{w_{ik} - z_{ik}\} \right]  \notag \\   
			& = \sum_{i:m_i>0} \left( \frac{1- \epsilon}{m_i} \right) \mathbf{E}(\mathbf{x}_{i}, \mathbf{y}_{i})  \notag  \\    
			& \leq m  \max_{i:m_i>0} \left( \frac{1- \epsilon}{m_{i}} \right) \mathbf{E}(\mathbf{x}_{i}, \mathbf{y}_{i}).  \label{finalmax}
		\end{align}
		The proof of Lemma \ref{lem:generalcontraction} is now complete.
	\end{proof}
	We now use Lemma \ref{lem:generalcontraction} to complete the proof of Theorem \ref{lem:wsm}. Let $i^*$ achieve the maximum in (\ref{finalmax}), that is,
	$i^* = \arg \max_i \frac{1}{m_i} \mathbf{E}_{\mathcal{G}_v,\mathcal{L}_{i,j_1,j_2}}(\mathbf{x}_{i}, \mathbf{y}_{i}) $. 
	Let $\mathcal{G}^1 = \mathcal{G}_v $, $\mathcal{L}^1 = \mathcal{L}_{i,j_1,j_2} $, $v^1 = v_{i^*}$ and $(\mathbf{x}^1, \mathbf{y}^1) = (\mathbf{x}_{i^*}, \mathbf{y}_{i^*})$.
	Lemma \ref{lem:generalcontraction} says that $\frac{1}{\Delta_{\mathcal{G}}(v)}  \mathbf{E}(\mathbf{x},\mathbf{y}) \leq (1 - \epsilon) \frac{1}{\Delta_{\mathcal{G}^1}(v^1)}  \mathbf{E}(\mathbf{x^1},\mathbf{y^1})$.
	Note that the graph list pair ($\mathcal{G}^1,\mathcal{L}^1$) 
	satisfies Assumption \ref{assume}. 
	 We can then apply apply Lemma \ref{lem:generalcontraction} to $v^1,\mathcal{G}^1,\mathcal{L}^1$ to obtain $v^2,\mathcal{G}^2,\mathcal{L}^2$ such that
	 $\frac{1}{\Delta_{\mathcal{G}}(v^1) } \mathbf{E}(\mathbf{x^1},\mathbf{y^1}) \leq (1 - \epsilon) \frac{1}{\Delta_{\mathcal{G}^2}(v^2)}  \mathbf{E}(\mathbf{x^2},\mathbf{y^2})$.
	 If we let $d = d(v,\partial \Psi)$, then
	applying Lemma \ref{lem:generalcontraction} successively $d$ times we obtain 
		\begin{align*}
			\frac{1}{\Delta_{\mathcal{G}}(v)} \mathbf{E}(\mathbf{x},\mathbf{y}) \leq (1-\epsilon) ^{d} \frac{1}{\Delta_{\mathcal{G}^d}(v^d)} 
			\mathbf{E}(\mathbf{x}^d,\mathbf{y}^d) \leq 2 \log \left(   \frac{q}{(1 - 1/\beta)^{\Delta}}  \right) {(1 - \epsilon)}^d,
		\end{align*}
		where the second inequality follows from Lemma \ref{lem:probub}.
		This gives for any $j \in L(v)$,
		\begin{align*}
			\log \left(  \frac{x_j}{y_j}\right) \leq \max_j \log \left(  \frac{x_j}{y_j}\right) \leq  \mathbf{E}(\mathbf{x},\mathbf{y}) \leq 2 \Delta ( \log q - \Delta \log(1 - 1/\beta )) (1-\epsilon) ^d
		\end{align*}
		Let $F = 2 \Delta ( \log q - \Delta \log(1 - 1/\beta ))$. The quantity $F$ depends only on the quantities defined in Assumption \ref{assume} and does not depend on the vertex $v$. 
		Let $d_0$ be large enough such that $\exp (F (1-\epsilon) ^{d_0} ) \leq 1 + 2 F(1-\epsilon) ^{d_0}$. Then for $d \geq d_0$, we have
		\begin{align*}
			\frac{\mathbf{P}(c(v) = j | C_1)}{ \mathbf{P}(c(v) = j | C_2)} \leq 1 + 2F e^{-\gamma d}.
		\end{align*}
		where $\gamma = -\log(1 - \epsilon)$. For $d \leq d_0$ 
		\begin{align*}
			\frac{\mathbf{P}(c(v) = j | C_1)}{ \mathbf{P}(c(v) = j | C_2)} \leq 1 + e^{F+\gamma d_0} e^{-\gamma d}.
		\end{align*}
		Taking $B = \max \{ e^{F+\gamma d_0} , 2F\}$ we get 
		\begin{align*}
			\frac{\mathbf{P}(c(v) = j | C_1)}{ \mathbf{P}(c(v) = j | C_2)} \leq 1 + B e^{-\gamma d}.
		\end{align*}
		The lower bound on the ratio of probabilities is obtained in a similar fashion. This completes the proof of Theorem \ref{lem:wsm}.
	
 \end{section}
 
 \begin{section}{Conclusion}
	In this paper, we proved that the strong spatial mixing for the list coloring problem holds for a general triangle free graph when for each vertex of the graph the size of its list is at least $\alpha \Delta(v) + \beta$ and 
	$\alpha > \alpha^* \approx 1.763$. This extends the previous results for strong spatial mixing of colorings for regular trees \cite{GeStefankovic11}
	 and for amenable triangle free graphs \cite{GoldbergMartinPaterson05}. An interesting next venture would be to use this long range 
	independence property to produce efficient approximation algorithms for counting colorings similar to \cite{GamarnikKatz08}. The main obstruction that we face here is that in order to prove contraction of 
	the recursion for $\alpha^* \approx 1.763$, we
	 need to use bounds on the probabilities mentioned in Lemma \ref{lem:probub}. This restricts our result to correlation decay with respect to distance in the graph theoretic sense
	instead of correlation decay in the computation tree, which was key to producing FPTAS in \cite{GamarnikKatz08}. 
	
	It would also be interesting to establish this result for smaller $\alpha$. It is conjectured that $\alpha = 1$ and $\beta = 2$ suffices but at the moment we are quite far from this result.
 \end{section}
 
\bibliographystyle{amsalpha}
\bibliography{Coloringbib}

\end{document}